\documentclass[12pt]{article}

\usepackage{amsmath,latexsym,amssymb,amsthm}
\usepackage{graphicx}
\usepackage{bm}
\usepackage{bbm}
\usepackage{epsfig}
\usepackage{authblk}

\newcommand{\one}{\mathbbm{1}} % One?

\renewcommand{\Re}{\operatorname{Re}} % Real part
\renewcommand{\Im}{\operatorname{Im}} % Imaginary part

\DeclareMathOperator{\diag}{diag}
\DeclareMathOperator{\sgn}{sgn}
\DeclareMathOperator{\spn}{span}
\DeclareMathOperator{\Tr}{Tr}

\newcommand{\GLnC}{\mathrm{GL}(n, \mathbb C)}

\newcommand{\Un}{\mathrm{U}(n)}

\newcommand{\Hermn}{\mathrm{Herm}(n)}

\newtheorem{theorem}{Theorem}[section]
\newtheorem{lemma}[theorem]{Lemma}

\newtheorem{corollary}[theorem]{Corollary}

\theoremstyle{definition}

\newtheorem{remark}[theorem]{Remark}

\numberwithin{equation}{section}

\title{Spherical functions approach to sums of random Hermitian matrices}

\author[1]{Arno B.J. Kuijlaars}
\author[2]{Pablo Rom\'an}

\affil[1]{Department of Mathematics, Katholieke Universiteit Leuven, Belgium,
arno.kuijlaars@kuleuven.be}
\affil[2]{CIEM, FaMAF, Universidad Nacional de C\'ordoba, Argentina,
roman@famaf.unc.edu.ar}

\date{}

\begin{document}

\maketitle

\begin{abstract}
We present an approach to sums of random Hermitian matrices via the
theory of spherical functions for the Gelfand pair $(\Un \ltimes \Hermn, \Un)$.
It is inspired by a similar approach of Kieburg and K\"osters for products of
random matrices. The spherical functions have determinantal expressions 
because of  the Harish-Chandra/Itzykson-Zuber integral formula. 
It leads to remarkably 
simple expressions for the  spherical transform and its inverse.
The spherical transform is applied to sums of unitarily invariant 
random matrices from polynomial ensembles and the subclass of 
polynomial ensembles of derivative type (in the additive sense), 
which turns out to be closed under  addition. 
We finally present additional detailed calculations for the 
sum with a random matrix from a Laguerre Unitary Ensemble.
\end{abstract}

\section{Introduction}

There is remarkable recent progress in the understanding of eigenvalues
and singular values of products of random matrices. This development
started with Akemann and Burda \cite{AkeBur2012} who found explicit
formulas for eigenvalues of products of complex Ginibre matrices
in terms of Meijer G-functions.
It was followed by the works \cite{AkeKieWei2013}, \cite{AkeIpsKie2013},
where the same was done for squared singular values of complex Ginibre
matrices, and again the formulas involve Meijer G-functions. 
Similar expressions were found for eigenvalues and squared singular values for
other products of truncated unitary matrices 
\cite{AdhRedRedSah2016, AkeBurKieNag2014}, and other random matrices.
The results for singular values were subsequently interpreted and extended as
transformations of polynomial ensembles in  
\cite{Kui2016, KuiSti2014, KieKuiSti2016}, see \cite{AkeIps2015} for a survey.

Recently, Kieburg and K\"osters \cite{KieKos2016a,KieKos2016b} presented
a natural harmonic analysis point of view on the results on 
products of random matrices in terms of  spherical functions associated 
with the Gelfand pair $(\GLnC,\Un)$. 
 It is the goal of this paper to give 
a similar interpretation for sums of random Hermitian matrices in terms 
of the Gelfand pair $(\Un \ltimes \Hermn, \Un)$.

In the rest of this introduction we summarize our results and state
the main theorem.
We consider probability density functions $f$ on the space $\Hermn$ of
$n \times n$ Hermitian matrices that are invariant under conjugation with 
unitary matrices. Thus $f(X) dX$ is a probability measure where
\begin{equation} \label{eq:dX} 
	dX = \prod_{j=1}^n dX_{jj} \, \prod_{j<k} d \Re X_{j,k} \, d \Im X_{j,k} 
	\end{equation}
is the flat Lebesgue measure on $\Hermn$, and we assume
\[ \int f(X) dX = 1, \qquad f(UXU^*) = f(X) \geq 0  \]
for every unitary matrix $U \in \Un$ and every Hermitian matrix $X \in \Hermn$. 
Then $f$ only depends on the eigenvalues
of $X$, say  $x_1, \ldots, x_n$, and we also write $f(x_1, \ldots, x_n)$.
By the Weyl integration formula we then have that
\begin{equation} \label{eq:Weyldensity} 
	\frac{\pi^{n(n-1)/2}}{\prod_{j=1}^n j!} \, f(x_1, \ldots, x_n) \, \Delta_n(x)^2
	 \end{equation}
is a probability density on $\mathbb R^n$ where
\[ \Delta_n(x) = \prod_{j<k} (x_k-x_j) = \det \left[ x_k^{j-1} \right]_{j,k=1}^n \]	
denotes the Vandermonde determinant.

The bounded spherical functions for $(\Un \ltimes \Hermn, \Un)$
are labelled by $s = (s_1, \ldots, s_n) \in \mathbb R^n$ and are given by
\begin{align} 
	\varphi_s(X) & = \int_{\Un} e^{i \Tr (SUXU^*)} dU 
	 = \left( \prod_{j=0}^{n-1} j! \right)
			\frac{\det\left[ e^{is_j x_k} \right]_{j,k=1}^n}
				{i^{n(n-1)/2} \Delta_n(s) \Delta_n(x) }  \label{eq:HCIZ}
	\end{align}
where $S = \diag(s_1, \ldots, s_n)$, $x = (x_1, \ldots, x_n)$  is
the vector of eigenvalues of the Hermitian matrix $X$, 
and $dU$ is the normalized Haar measure
on the unitary group $\Un$. This follows from  more general
results in \cite{BenJenRat1990}, see also \cite{Far2010} and the
discussion in Section 2 below.  The second identity in \eqref{eq:HCIZ} 
is the well-known Harish-Chandra/Itzykson-Zuber formula, see e.g.\ 
\cite[Proposition 11.6.1]{For2010}.

The corresponding spherical transform is $f \mapsto \widehat{f}$ 
where
\begin{equation} \label{eq:fhatX} 
	\widehat{f}(s) = \int f(X) \varphi_s(-X) \, dX 
	\end{equation}
with the integral over the set of $n \times n$ Hermitian matrices.	
In terms of an integral over eigenvalues this is by \eqref{eq:HCIZ}
and the Weyl integration formula
\begin{equation} \label{eq:fhat}
	\widehat{f}(s)  
	= \frac{(\pi i)^{n(n-1)/2}}{n! \, \Delta_n(s)}
		\int_{\mathbb R^n} f(x) 
		\, \det\left[e^{- i s_j x_k} \right]_{j,k=1}^n \,
		\Delta_n(x) \,	dx
			\end{equation}
where $x = (x_1, \ldots, x_n)$ and $dx = dx_1 \cdots dx_n$.
The prefactor in \eqref{eq:fhat} is such that 
\[ \widehat{f}(0,\ldots, 0) = \int f(X) \, dX = 1, \]
which should hold since $f$ is a probability density function
and $\varphi_{(0,\ldots,0)}(X) = 1$ for every $X \in \Hermn$.

In Section 2 below we show that there is an inversion formula
\begin{equation} \label{eq:finverse} 
	f(x) 
	=  \frac{(\pi i)^{-n(n-1)/2}}{(2\pi)^n n! \Delta_n(x)}
		\int_{\mathbb R^n} 
		\widehat{f}(s) \det\left[ e^{i s_j x_k} \right]_{j,k=1}^n
			\Delta_n(s) \, ds, 
			\end{equation}
which basically follows from the multidimensional inverse Fourier transform.

Our main result is the following.
\begin{theorem} \label{thm:main}
Suppose that $X$ and $Y$ are independent unitarily invariant random 
Hermitian matrices with probability densities $f_X$ and $f_Y$, respectively.
Let $f_{X+Y}$ be the probability  density of the sum $X+Y$. Then
\begin{equation} \label{eq:fXY} 
	\widehat{f}_{X+Y} = \widehat{f}_X \cdot \widehat{f}_Y 
	\end{equation}
and 
\begin{equation} \label{eq:fXYinverse}
	f_{X+Y}(x) 
	=  \frac{(\pi i)^{-n(n-1)/2}}{(2\pi)^n n! \Delta_n(x)}
		\int_{\mathbb R^n} 
		\widehat{f}_X(s) \, \widehat{f}_Y(s) \, 
		\det\left[ e^{i s_j x_k} \right]_{j,k=1}^n
			\Delta_n(s) \, ds. 
			\end{equation}
\end{theorem}
The formula \eqref{eq:fXYinverse} is of course the inversion formula
\eqref{eq:finverse} applied to \eqref{eq:fXY}.

The formulas \eqref{eq:fXY} and \eqref{eq:fXYinverse} are analogues of
familiar properties of the usual Fourier transform. We emphasize that
they hold for random Hermitian matrices. The situation for real
symmetric matrices is not so nice, because of the lack of a 
Harish-Chandra/Itzykson-Zuber formula  \eqref{eq:HCIZ} for the 
corresponding integral over the orthogonal group.

The proof of Theorem \ref{thm:main} is in Section 2. We discuss
the Gelfand pair $(\Un \ltimes \Hermn, \Un)$ and its relevance for
sums of random Hermitian matrices. Then we recall the general concept
of a Gelfand pair, the notion of spherical functions and spherical
transform. Then we specialize again to the case $(\Un \ltimes \Hermn, \Un)$
and we show that the bounded spherical functions are given by
the functions \eqref{eq:HCIZ}.
We compute the spherical transform \eqref{eq:fhat} and its 
inverse \eqref{eq:finverse} and then finish the proof
of Theorem \ref{thm:main}.

In Section 3 we compute the spherical transform $\widehat{f}$
in certain situations. The result is explicit for the
probability density functions of the GUE and LUE random matrix ensembles.
For a polynomial ensemble the spherical transform is a ratio
of determinants \eqref{eq:fPEhat} with a Vandermonde determinant in
the denominator. From Theorem \ref{thm:main} we then 
find that the sum of a polynomial ensemble with a GUE or LUE matrix
is again a polynomial ensemble, see Corollaries \ref{cor:GUEsum} and
\ref{cor:LUEsum}. The result for GUE was already noted in 
\cite{ClaKuiWan2015}, while it is a new result for the sum with an LUE matrix.
We end Section 3 with polynomial ensembles of derivative type,
following the similar notion introduced in \cite{KieKos2016a,KieKos2016b}
in a multiplicative setting.

In Section 4 we provide more information on the sum of a polynomial
ensemble with an LUE matrix. We give a second proof of 
Corollary \ref{cor:LUEsum}, as well as tranformation results for the
correlation kernel and the biorthogonal functions that are associated
with a polynomial ensemble.

\begin{remark}
The spherical functions for the Gelfand pair $(\GLnC, \Un)$
can be written as
\[ \varphi_s(X) = \left(\prod_{j=0}^{n-1} j! \right) 
	 \frac{\det \left[ x_j^{s_k} \right]_{j,k=1}^n}{\Delta_n(s) \Delta_n(x)},
	 	\qquad X \in \GLnC, \]
where $x_1, \ldots, x_n$ are the eigenvalues of $X^*X$,
see \cite{GelNai1957,KieKos2016b}.
These functions are used in \cite{KieKos2016b} for products
of random matrices, in a similar way as \eqref{eq:HCIZ} will be used
for sums of random matrices in this paper. 
\end{remark}

\section{Spherical functions}

\subsection{The Gelfand pair}
The Gelfand pair $(\Un \ltimes \Hermn, \Un)$ consists of the semidirect
product $G = \Un \ltimes \Hermn$ of the unitary group with the real
vector space $\Hermn$ of complex Hermitian $n \times n$ matrices
and the compact subgroup $K =  \Un \times \{0\} \simeq \Un$.
A unitary matrix $U \in \Un$ acts on $\Hermn$ by conjugation 
$A \mapsto UA U^*$. The composition law on $G$ is
\begin{equation} \label{eq:grouplaw} 
	(U_1, A) \cdot (U_2, B) = (U_1U_2, A + U_1 B U_1^*)
	\end{equation}
for $U_1, U_2 \in \Un$ and $A,B \in \Hermn$.

A function $f : G \to \mathbb C$ is bi-$K$-invariant if it is invariant under
left and right multiplication with elements of $K$. In our situation it means
that $f(U,A)$ only depends on the eigenvalues of the Hermitian matrix $A$.
For a bi-$K$-invariant function $f$ we therefore simply write $f(A)$ instead 
of $f(U,A)$,
or even $f(x_1, \ldots, x_n)$ where $x_1, \ldots, x_n$ are the eigenvalues of $A$,
see also the introduction.

There is a convolution product for functions on $G$, which for bi-$K$-invariant
functions reduces to (it is a simple verification)
\begin{equation} \label{eq:convolution} 
	(f \ast g)(A) = \int_{\Hermn} f(X) g(A-X) \, dX
	\end{equation}
where $dX$ is the Lebesgue measure on $\Hermn$, see  \eqref{eq:dX}.
From \eqref{eq:convolution} it is obvious that $f \ast g = g \ast f$ for 
bi-$K$-invariant functions, which is the property that defines a Gelfand 
pair $(G,K)$. 

The connection to random matrices is in the following corollary. 
\begin{corollary} \label{cor:fsum}
If $f_X$ and $f_Y$  are the probability densities for independent 
bi-$K$-invariant random Hermitian matrices $X$ and $Y$
then $f_X \ast f_Y = f_{X+Y}$ is the probability density for 
the sum $X + Y$.
\end{corollary}

\subsection{Spherical functions: general concepts}
We follow the exposition of van Dijk \cite{vDij2009}.

In this subsection  we let $G$ be a general locally compact group with a 
compact subgroup $K$. Later we will specialize it to the case
$G = \Un \ltimes \Hermn$ and $K= \Un$.
In the general setting we use lower case letters for elements 
of $G$ and $K$. Let $dx$ be left Haar measure on $G$ and let $dk$
be normalized Haar measure on $K$.

Let $C_c^\#(G)$ be the space
of continuous compactly supported complex valued functions on $G$
that are bi-invariant with respect to $K$. The group structure
on $G$  gives rise to a convolution product 
\[ (f \ast g)(x) = \int_G f(y) g(y^{-1}x) \, dy \]
for functions $f$ and $g$ on $G$. When restricted to $C_c^\#(G)$
it turns $C_c^\#(G)$ into a convolution algebra.
The pair $(G,K)$ is called a Gelfand pair if $C_c^\#(G)$ is commutative.

A spherical function for a Gelfand pair $(G,K)$ is a nonzero continuous 
function $\varphi$ on $G$ such that
\begin{equation} \label{eq:spherical}
	\int_K \varphi(xky) dk = \varphi(x) \varphi(y), \qquad x, y \in G.
\end{equation}
It is equivalent to saying that the functional $\chi$ on $C_c^\#(G)$ defined by
\[ \chi(f) = \int_G f(x) \varphi(x^{-1}) dx \]
is a non-trivial character, i.e.,
\[ \chi(f \ast g) = \chi(f) \chi(g), \]
see \cite[Proposition 6.1.5]{vDij2009}. A spherical function $\varphi$ is 
bi-$K$-invariant  and $\varphi(e) = 1$, where $e$ is the unit element in $G$. 

A locally integrable function $\varphi : G \to \mathbb C$ is positive-definite if
\[ \int_G \int_G \varphi(x^{-1}y) f(x) \overline{f(y)} dx dy \geq 0 \]
for every continuous function $f$ with compact support on $G$.
If $\pi$ is a unitary representation of $G$ on a Hilbert space $H$, and 
$\varepsilon \in H$,
then $x \mapsto \langle \varepsilon, \pi(x) \varepsilon \rangle$ is
a bounded continuous positive-definite function on $G$, and every bounded
continuous positive-definite function is obtained this way. We may assume in
addition that $\varepsilon$ is a cyclic vector, see \cite[Remark 5.1.7]{vDij2009}.

A continuous positive-definite function $\varphi$ on $G$ that is bi-$K$-invariant
 with $\varphi(e) = 1$
is a spherical function if and only if the associated unitary representation is
irreducible \cite[Theorems 5.3.2 and 6.2.5]{vDij2009}.

Let $Z$ be the set of positive-definite spherical functions. Such functions are
 automatically continuous and bounded.
Then the spherical transform $\widehat{f}$ of a function $f \in L^1(G)^\#$ 
is defined as
\begin{equation} \label{eq:fhatgeneral} 
	\widehat{f} : Z \to \mathbb C : \quad \varphi \in Z \mapsto 
	\widehat{f}(\varphi) = \int_G f(x) \varphi(x^{-1}) dx, 
	\end{equation}
see \cite[Definition 6.4.3]{vDij2009} where it is called the Fourier transform.
There is a natural topology on $Z$, which is locally compact.
Then by \cite[page 84]{vDij2009},  $\widehat{f}$ is a continuous function on $Z$
that vanishes at infinity, and
\[  | \widehat{f}(\varphi)| \leq  \int_G |f(x)| dx, \qquad \varphi \in Z, \]
the transformation $f \mapsto \widehat{f}$ is linear with
\begin{equation} \label{eq:fgconv} 
	\widehat{f \ast g} = \widehat{f} \, \widehat{g}, \qquad f,g \in L^1(G)^\#.
	\end{equation}

There is an inversion formula according to which we can recover $f$ from 
$\widehat{f}$. Namely, there is a unique measure $\nu$ on $Z$ 
(sometimes called Plancherel measure)
such that for bi-$K$-invariant functions $f$,
\[ f(x) = \int_Z \varphi(x) \widehat{f}(\varphi) d \nu(\varphi), \]
and there is a Plancherel formula
\[ \int_G |f(x)|^2 dx = \int_Z | \widehat{f}(\varphi)|^2 d\nu(\varphi), 
\qquad f \in L^2(G)^\#, \]
see \cite[Theorems 6.4.5 and 6.4.6]{vDij2009}. 

\subsection{Spherical functions in special case}
Suppose the Gelfand pair takes the form $(K \ltimes N, K)$
where $K$ is a compact group that acts on $N$.
Then bi-$K$-invariant functions on $G = K \ltimes N$ are naturally
identified with functions on $N$ that are invariant under the 
action of $K$. 
This situation was considered by Benson, Jenkins and Ratcliff
\cite{BenJenRat1990}. When $N$ is a nilpotent Lie group,
they characterized the spherical functions as follows.
\begin{lemma} \cite[Lemma 8.2 and Corollary 8.4]{BenJenRat1990}  \label{lem:BJR}
Suppose $\varphi$ is a bounded spherical function on $N$, where $N$ is
a nilpotent Lie group. Then $\varphi$ is positive-definite, and
there exist an irreducible unitary representation $\pi$ of $N$ 
on a Hilbert space $H_{\pi}$ 
and a unit vector $\xi \in H_{\pi}$ such that
\begin{equation} \label{eq:BJRlemma} 
	\varphi(x) = \int_K  \langle \pi(k \cdot x) \xi, \xi \rangle dk 
	\end{equation}
for each $x \in N$.
\end{lemma}
	
Lemma \ref{lem:BJR} applies to the Gelfand pair $(\Un \ltimes \Hermn, \Un)$
since the vector space $\Hermn$ is an abelian group, and thus nilpotent.
The pairing $\langle A, B \rangle = \Tr (AB)$  is a real inner product 
on $\Hermn$, and  all irreducible unitary representations of $\Hermn$ are 
given by $\pi_S:\Hermn\to \mathbb{C}$, where
\[ \pi_S: X\mapsto e^{i \langle S, X \rangle} = e^{i \Tr(SX)},\qquad 
S \in \Hermn. \]
From Lemma \ref{lem:BJR} we thus obtain that all bounded positive-definite 
spherical functions are given by
	\begin{align} \nonumber
 \varphi_S(A) & = \int_{\Un} \pi_S(UAU^*) dU \\
 	&  = \int_{\Un} e^{i \Tr (S U A U^*)} dU,
 	\qquad A \in \Hermn, \label{eq:sphericalintegral}
 	\end{align}
where $dU$ denotes the normalized Haar measure on $\Un$, and $S \in \Hermn$. 

\begin{remark}
In fact we can  verify the property \eqref{eq:spherical} directly from
 \eqref{eq:sphericalintegral}, for any fixed $n \times n$ matrix $S$.
If $x = (U_1,A) \in G$, $k = (V,0) \in K$ and $y= (U_2,B) \in G$, then by the
 composition rule \eqref{eq:grouplaw} we have
 \[ xky = (U_1VU_2, A + U_1V B V^* U_1^*). \]
 The definition \eqref{eq:sphericalintegral} only uses the Hermitian part of $xky$
 and we have for Hermitian matrices $A$ and $B$,
 \begin{multline} \label{eq:sphericalcheck}
 \int_{\Un} \varphi_S(A+U_1V B V^*U_1^*) dV  = 
 \int_{\Un} \int_{\Un} e^{i \Tr (S U (A+U_1V B V^*U_1^*) U^*)} dU dV \\
  = \left(\int_{\Un} e^{i \Tr SUAU^*} dU \right)
 	\left(\int_{\Un} \int_{\Un} e^{i \Tr SUU_1 V B V^*U_1^*U^*} dV dU \right) 
 \end{multline}
 where we used the linearity of the trace and Fubini's theorem.
 By the invariance of Haar measure and \eqref{eq:sphericalintegral}
\[  	\left(\int_{\Un} \int_{\Un} e^{i \Tr SUU_1 V B V^*U_1^*U^*} dV dU \right) = \varphi_S(B)
\] 
and by \eqref{eq:sphericalcheck} we find indeed
\[ \int_{\Un} \int_{\Un} e^{i \Tr (S U (A+U_1V B V^*U_1^*) U^*)} dU dV   =
	 \varphi_S(A) \varphi_S(B). \]
It is also clear from \eqref{eq:sphericalintegral} that $\varphi_S(0) = 1$ and thus
$\varphi_S$ is a spherical function. It is bounded if and only if $S \in \Hermn$.
\end{remark}

For every matrix $S\in \Hermn$ there exists $U\in \Un$ such that 
$USU^*=\diag(s_1,\ldots,s_n)$ where $s_j\in \mathbb{R}$ for $j =1, \ldots, n$. 
By \eqref{eq:sphericalintegral} and the invariance of the
Haar measure, the spherical function only depends on the eigenvalues of $S$
and we write $\varphi_s$ instead of $\varphi_S$ where $s = (s_1,\ldots, s_n)$.
Moreover, $\varphi_{(s_1,\ldots,s_n)}=\varphi_{(s_{\sigma(1)},
\ldots,s_{\sigma(n)})}$ for any permutation $\sigma\in S_n$. 
Therefore the set $Z$ of positive-definite spherical functions can be 
identified with $\mathbb{R}^n \slash S_n$.
 
By the Harish-Chandra/Itzykson-Zuber formula the spherical function 
\eqref{eq:sphericalintegral} takes the determinantal form \eqref{eq:HCIZ}.
The formula should be understood in a limiting sense
if some of the $x_j$'s and/or some of the $s_j$'s coincide.

\subsection{Proof of Theorem \ref{thm:main}}

The spherical transform \eqref{eq:fhatgeneral} of a  function $f \in L^1(G)^\#$ 
is in our special case $(G,K) = (\Un \ltimes \Hermn, \Un)$,
\begin{equation} \label{eq:sphericaltrans} 
	\widehat{f}(s) = \int_{\Hermn} f(A) \varphi_s(-A) \, dA,
\end{equation} 
which we view as a function on $\mathbb R^n$ that is invariant under
permutation of coordinates (instead of a function on $Z = \mathbb R^n \slash S_n$).
If $f$ has compact support, then \eqref{eq:sphericaltrans} 
is defined for all $s \in \mathbb C^n$. 
As an integral over eigenvalues the spherical transform 
\eqref{eq:sphericaltrans} it gives us \eqref{eq:fhat} by the Weyl
integration formula.
Then \eqref{eq:fXY} in Theorem \ref{thm:main}  follows because of Corollary
 \ref{cor:fsum} and \eqref{eq:fgconv}.
	
Next, we expand the determinant $\det\left[e^{- s_j x_k}\right]_{j,k=1}^n$ 
to obtain from \eqref{eq:fhat} that 
\begin{align*}
	\widehat{f}(s) \Delta_n(s) = \frac{(\pi i)^{n(n-1)/2}}{n!} 
	\sum_{\sigma \in S_n} \sgn(\sigma) 	
	\int_{\mathbb R^n} f(x) e^{- i s_j x_{\sigma(j)}} \Delta_n(x) \, dx, 
	\end{align*}
where the sum is over permutations $\sigma \in S_n$.
Because $f(x) = f(x_1, \ldots, x_n)$ is invariant
under permutations of coordinates, while $\Delta_n(x)$ changes sign for odd
 permutations, each permutation $\sigma$ has the same contribution. Thus 
\[  \widehat{f}(s) \Delta_n(s) = 
	(\pi i)^{n(n-1)/2} \int_{\mathbb R^n} f(x)  e^{-i s_j x_j} 
	\, \Delta_n(x) \, dx \]
which is the usual $n$-dimensional Fourier transform of $f(x) \Delta_n(x)$. 
Thus by Fourier inversion
\begin{align*} 
	f(x) \Delta_n(x) & =  \frac{1}{(2 \pi)^n (\pi i)^{n(n-1)/2}}  
	\int_{\mathbb R^n} 
	 \widehat{f}(s)   e^{ i s_j x_j} \, \Delta_n(s) \, ds 
		\end{align*}
	with $ds = ds_1 \cdots ds_n$.
Now $\widehat{f}(s)$ is invariant under permutations of $s_1, \ldots, s_n$. 
Then by similar argument as above, we can write
\begin{align*} f(x) \Delta_n(x) & =  
	\frac{1}{(2\pi)^n (\pi i)^{n(n-1)/2} n!}  
	\sum_{\sigma \in S_n} \sgn(\sigma) 	\int_{\mathbb R^n} 
		\widehat{f}(s)  e^{i s_{\sigma(j)} x_j} \, \Delta_n(s) \, ds \\
	& = \frac{1}{(2\pi)^n (\pi i)^{n(n-1)/2} n!}  
	\int_{\mathbb R^n} \widehat{f}(s) 
	\det\left[ e^{s_k x_j} \right]_{j,k=1}^n \, \Delta_n(s) \, ds  
\end{align*}
which is \eqref{eq:finverse}.
This proves the inversion formula \eqref{eq:fXYinverse} 
and the proof of Theorem \ref{thm:main} is complete.
\qed

\begin{remark}
Recalling the expression \eqref{eq:HCIZ} of the spherical function, 
and writing $f(A)$ instead of $f(x)$, we have
\[ f(A) = \frac{1}{(2\pi)^n (\pi i)^{n(n-1)/2}  \prod_{j=0}^{n} j!}  
\int_{\mathbb R^n} \widehat{f}(s)  \varphi_s(A) \, \Delta_n(s)^2 \, ds. \]
So the Plancherel measure on $Z$ is proportional to $\Delta_n(s)^2 \, ds$.
\end{remark}

\section{Computation of special cases}

The integral \eqref{eq:fhat} can be evaluated explicitly in certain cases. 
The Andreief formula \cite{DeiGio2009}
\begin{equation} \label{eq:Andreief} 
	\int \det \left[f_j(x_k) \right]_{j,k=1}^n 
		\det \left[g_j(x_k) \right]_{j,k=1}^n dx_1 \cdots dx_n  
	= n! \det \left[ \int f_j(x) g_k(x) dx \right]_{j,k=1}^n 
	\end{equation}
will be useful in the computations. 

\subsection{Gaussian Unitary Ensemble}
The density of the Gaussian Unitary Ensemble (GUE) is 
\begin{equation} \label{eq:fGUE} 
	f_{GUE}(X) = \frac{1}{2^{n/2} \pi^{n^2/2}} e^{- \frac{1}{2} \Tr X^2}.
	\end{equation}
In terms of eigenvalues  we have 
\[ f_{GUE}(x_1, \ldots, x_n) = 
	\frac{1}{2^{n/2} \pi^{n^2/2}} \prod_{k=1}^n e^{-\frac{1}{2} x_k^2}, \]
and then by \eqref{eq:fhat} and the Andreief formula \eqref{eq:Andreief},
\begin{equation} \label{eq:fGUEhat1} 
	\widehat{f}_{GUE}(s)
	=  \frac{ i^{n(n-1)/2}}{(2\pi)^{n/2} \Delta_n(s)} 
	 \det \left[\int_{-\infty}^{\infty} e^{-\frac{1}{2}x^2} e^{-i s_jx} x^{k-1} dx
	  \right]_{j,k=1}^n. \end{equation} 
The integrals can be evaluated, since
\[ \frac{1}{\sqrt{2\pi}} \int_{-\infty}^{\infty} 
	e^{-\frac{1}{2}x^2} e^{-i sx}   dx = e^{-\frac{1}{2} s^2} \]
and for $k = 2, \ldots, n$,
\begin{align*} 
	\frac{1}{\sqrt{2\pi}} \int_{-\infty}^{\infty} 
	e^{-\frac{1}{2}x^2} e^{-i sx} s^{k-1}  dx   = 
		\left(i \frac{d}{ds} \right)^{k-1} \left( e^{-\frac{1}{2} s^2} \right)
		= (-i)^{k-1} P_{k-1}(s) e^{-\frac{1}{2} s^2} \end{align*}
for a certain  monic polynomial $P_{k-1}$ of degree $k-1$.
Inserting this into \eqref{eq:fGUEhat1} we obtain
\begin{align}  \nonumber
	\widehat{f}_{GUE}(s) & = \frac{ i^{n(n-1)/2}}{\Delta_n(s)} 
	\det \left[ (-i)^{k-1} P_{k-1}(s_j) e^{-\frac{1}{2} s_j^2} \right]_{j,k=1}^n \\
	& = \nonumber
	\frac{ i^{n(n-1)/2}}{\Delta_n(s)} \left(\prod_{k=1}^n (-i)^{k-1} \right)
	\det \left[ P_{k-1}(s_j) \right]_{j,k=1}^n
		\left(\prod_{j=1}^n e^{-\frac{1}{2} s_j^2} \right)  \\
		& =  \label{eq:fGUEhat}
		\prod_{j=1}^n  e^{-\frac{1}{2} s_j^2},
\end{align} 
where we used the fact that for any sequence of monic polynomials 
$(P_j)_{j=0}^{n-1}$ with
$\deg P_j = j$, one has $\det \left[ P_{k-1}(s_j) \right]_{j,k=1}^n = \Delta_n(s)$.

\subsection{Laguerre Unitary Ensemble}
The Laguerre Unitary Ensemble (LUE) with parameter $\alpha > -1$ is given by
the probability density 
\begin{equation} 
	\label{eq:fLUE} f_{LUE}(L) = \frac{(\det L)^{\alpha}}
	{\pi^{n(n-1)/2} \prod_{j=1}^n \Gamma(\alpha +j)} 
	\, e^{-\Tr L} 
\end{equation}
on the set of positive definite Hermitian matrices $L$. Thus 
\[ f_{LUE}(x_1, \ldots, x_n) =  \pi^{-n(n-1)/2} 
	\prod_{j=1}^n \Gamma(\alpha+j) \,
	 \prod_{k=1}^n x_k^{\alpha} e^{-x_k} \, \one_{x_k \geq 0}. \]
By \eqref{eq:fhat} and the Andreief formula \eqref{eq:Andreief},
\begin{equation} \label{eq:fLUEhat1} 
	\widehat{f}_{LUE}(s)
	=  \frac{ i^{n(n-1)/2}}{\prod_{j=1}^n \Gamma(\alpha+j)^{-1} 
	 \Delta_n(s)} 
	 \det \left[ \int_{0}^{\infty} x^{\alpha+k-1} e^{- x} e^{-i s_jx}  dx
	  \right]_{j,k=1}^n. \end{equation} 
We compute the integrals
\[ \int_{0}^{\infty} x^{\alpha+k-1} e^{- x} e^{-i s x} dx = 
	\frac{\Gamma(\alpha+k)}{(1+is)^{\alpha+k}}  \]
and \eqref{eq:fLUEhat1} simplifies to
	\begin{align} \nonumber
	\widehat{f}_{LUE}(s)
		& = \frac{i^{n(n-1)/2}}{\Delta_n(s)} 
	 	\det \left[ \frac{1}{(1+is_j)^{\alpha+k}} \right]_{j,k=1}^n \\
	 	& = \frac{ i^{n(n-1)/2}}{ \Delta_n(s)} 
	 		\det \left[ (1+is_j)^{n-k} \right]_{j,k=1}^n
	 			\prod_{j=1}^n \frac{1}{(1+is_j)^{\alpha+n}}.
	 	\label{eq:fLUEhat2} 
	 	\end{align}
The remaining determinant in \eqref{eq:fLUEhat2} is 
$(-i)^{n(n-1)/2} \Delta_n(s)$	 	
and we find the following spherical transform for the LUE density 
\begin{equation} \label{eq:fLUEhat} 
	\widehat{f}_{LUE}(s) = 
	\prod_{j=1}^n \frac{1}{(1+i s_j)^{\alpha+n}}. 
	\end{equation}

\subsection{Polynomial ensemble}
A polynomial ensemble \cite{KuiSti2014,Kui2016} is a probability density 
on $\mathbb R^n$ of the form
\begin{equation} \label{eq:polensemble} 
	\frac{1}{Z_n} \Delta_n(x) \det \left[ w_k(x_j) \right]_{j,k=1}^n 
	\end{equation}
for some given functions $w_1, \ldots, w_n$, and a certain normalization
constant $Z_n$. 
If $X$ is a random Hermitian matrix, then we  write 
\[ X \sim PE(w_1,\ldots, w_n) \]
if the induced probability density on the eigenvalues is of the form 
\eqref{eq:polensemble}. If $X$ is unitarily invariant with
a probability density $f$, then in view of \eqref{eq:Weyldensity}
this means that
\begin{equation} \label{eq:polensemble2} 
	f(A) = f(x_1, \ldots, x_n) = \frac{1}{Z_n'} 
		\frac{\det \left[ w_k(x_j) \right]_{j,k=1}^n}{\Delta_n(x)}.
	\end{equation}

Using \eqref{eq:polensemble2} in \eqref{eq:fhat} we find for
the spherical transform of a polynomial ensemble
\begin{equation} \label{eq:fPEhat} 
\widehat{f}(s_1, \ldots, s_n) = 
	\frac{1}{Z_n'' \Delta_n(s)} \det\left[ \int_{-\infty}^{\infty} w_k(x) 
	e^{-i s_jx} dx \right]_{j,k=1}^n,
	\end{equation}
where we also used the Andreief formula \eqref{eq:Andreief}.
The normalization constant is such that $\widehat{f}(0,\ldots,0) = 1$,
which by l'Hopital's rule means that
\[ Z_n'' = \frac{(-i)^{n(n-1)/2}}{\prod_{j=0}^{n-1} j!}
	\det\left[ \int_{-\infty}^{\infty} w_k(x) x^{j-1} dx \right]_{j,k=1}^n. \]

We now recover the following result that was
proved in a different way by Claeys, Kuijlaars and Wang 
\cite[Theorem 2.1]{ClaKuiWan2015}. 
\begin{corollary} \label{cor:GUEsum}
	Let $X \sim PE(f_1, \ldots, f_n)$ be an $n \times n$ unitarily 
	invariant random Hermitian matrix for certain functions $f_1, \ldots, f_n$. 
	Let $Y$ be an $n \times n$  GUE matrix, independent of $X$. 
	Then $X+Y \sim PE(g_1, \ldots, g_n)$
	where
	\begin{equation} \label{eq:gkGUE} 
		g_k(y) = \int_{-\infty}^{\infty} e^{-\frac{1}{2}x^2} f_k(y-x) dx. 
		\end{equation}
\end{corollary}
\begin{proof}
We have by \eqref{eq:fPEhat}
\[ \widehat{f}_X(s) \propto 
	\frac{1}{\Delta_n(s)} \det\left[ \int_{-\infty}^{\infty} f_k(x) 
	e^{-i s_jx} dx \right]_{j,k=1}^n
	\]
and by \eqref{eq:fGUEhat}
\[ \widehat{f}_Y(s) = \prod_{j=1}^n e^{-\frac{1}{2} s_j^2}. \]
Then by Theorem \ref{thm:main}
\begin{align} \nonumber 
	\widehat{f}_{X+Y}(s) & \propto
	\frac{1}{\Delta_n(s)} \det\left[ \int_{-\infty}^{\infty} f_k(x) 
	e^{-i s_jx} dx \right]_{j,k=1}^n \, \prod_{j=1}^n e^{-\frac{1}{2} s_j^2} \\
	& = \nonumber
	\frac{1}{\Delta_n(s)} 
	\det\left[ e^{-\frac{1}{2} s_j^2} \int_{-\infty}^{\infty} f_k(x) 
	e^{-i s_jx} dx \right]_{j,k=1}^n \\
	& \propto \label{eq:GUEsum1}
	\frac{1}{\Delta_n(s)} 
	\det\left[ \left( \mathcal F [e^{-\frac{1}{2} x^2}] \right)(s_j)
		\left(\mathcal F f_k \right)(s_j)  \right]_{j,k=1}^n 
	\end{align}
where $\mathcal F$ is the Fourier transform 
\begin{equation} \label{eq:FT}
	(\mathcal F w)(s) = 
		\int_{-\infty}^{\infty} w(x) e^{-isx} dx.
\end{equation}
The function $g_k$ from \eqref{eq:gkGUE} is the convolution of $f_k$ with
$x \mapsto e^{-\frac{1}{2}x^2}$ and by elementary properties of the
Fourier transform it follows from \eqref{eq:GUEsum1}
\begin{align*}
	\widehat{f}_{X+Y}(s) \propto
	\frac{1}{\Delta_n(s)} 
	\det\left[ \mathcal F g_k(s_j) \right]_{j,k=1}^n 
\end{align*}
which by \eqref{eq:fPEhat} and the injectivity of the spherical transform
means that $X+Y \sim PE(g_1, \ldots, g_n)$ as claimed.
\end{proof}

In the same way we can combine \eqref{eq:fPEhat} and \eqref{eq:fLUEhat}
and find the following new result about addition of an LUE matrix. 
\begin{corollary} \label{cor:LUEsum}
	Let $X$ be an $n \times n$ unitarily 
	invariant random Hermitian matrix such that $X \sim PE(f_1, \ldots, f_n)$
	for certain functions $f_1, \ldots, f_n$. 
	Let $L$ be an $n \times n$  LUE matrix with parameter $\alpha$, 
	independent of $X$. 
	Then $X+L \sim PE(g_1, \ldots, g_n)$
	where
	\begin{equation} \label{eq:gkLUE} 
		g_k(y) = \int_{0}^{\infty} x^{\alpha+n-1} e^{-x} f_k(y-x) \, dx. 
	\end{equation}
\end{corollary}
\begin{proof}
The proof is exactly the same as the proof of Corollary \ref{cor:LUEsum}.
We only use \eqref{eq:fLUEhat} instead of \eqref{eq:fGUEhat} and the
fact that
\[  \mathcal F[x^{\alpha+n-1} e^{-x} \one_{x\geq 0}](s) = 	
		  \frac{\Gamma(\alpha+n)}{(1+is)^{n}}. \]
where $\mathcal F$ is the Fourier transform as in \eqref{eq:FT}.
\end{proof}
There is an interesting alternative proof of Corollary 
\ref{cor:LUEsum}, along the lines of the proof in \cite{ClaKuiWan2015} 
of Corollary \ref{cor:GUEsum}, which we give in Section 4.

\subsection{Polynomial ensemble of derivative type}

Polynomial ensembles of derivative type were introduced by
Kieburg and K\"osters in \cite{KieKos2016a,KieKos2016b} in the 
connection with products of random matrices.
There is an analogous notion that is relevant for sums of random
matrices, and we call it polynomial ensemble of derivative type (in the
additive sense).

The polynomial ensemble \eqref{eq:polensemble} is of derivative type
(in the additive sense) if 
\[ \spn\{w_1, \ldots, w_n\} = \spn\{ w^{(k)} \mid k = 0, \ldots, n-1 \} \]
for some function $w$. In that case, we may use elementary  column
transformations to the determinant in \eqref{eq:polensemble} to
pass from $w_1, \ldots, w_n$ to $w, w', \ldots, w^{(n-1)}$ (with a possibly
different normalization constant).
The corresponding probability density \eqref{eq:polensemble2}  
then is
\begin{equation} \label{eq:polensembleDT} 
	f(A) = f(x_1, \ldots, x_n) \propto  \frac{\det \left[ w^{(k-1)}(x_j) \right]_{j,k=1}^n}{\Delta_n(x)}
	\end{equation}
and the density on eigenvalues is 
\begin{equation} \label{eq:polensembleDT2} 
	 \frac{1}{Z_n} \, \Delta_n(x) \, \det\left[ w^{(k-1)}(x_j)\right]_{j,k=1}^n
	\end{equation}
The  spherical transform \eqref{eq:fPEhat} simplifies
in this case since
\[ \int_{-\infty}^{\infty} w^{(k-1)}(x) e^{-i sx} dx = 
	(is)^{k-1} \int_{-\infty}^{\infty} w(x) e^{-isx} dx \]
and
\begin{align} \nonumber
	\widehat{f}(s_1, \ldots, s_n) & \propto
		\frac{1}{\Delta_n(s)} \det \left[ (is_j)^{k-1} \right]_{j,k=1}^n
			\prod_{j=1}^n \int_{-\infty}^{\infty} w(x) e^{-i s_jx} dx \\
			& \propto \prod_{j=1}^n \int_{-\infty}^{\infty} w(x) e^{- is_jx} dx.
 \label{eq:fDPEhat1} 
\end{align}	
The proportionality constant follows from the property 
$\widehat{f}(0, \ldots, 0) = 1$. Thus for a polynomial ensemble of
derivative type (in the additive sense) the spherical transform factorizes as 
\begin{equation} \label{eq:fDPEhat} 
	\widehat{f}(s_1, \ldots, s_n) 
	= 
	\frac{1}{(\int_{-\infty}^{\infty} w(x) dx)^n} \prod_{j=1}^n (\mathcal F w)(s_j) 
\end{equation}
where $\mathcal F$ is again the Fourier transform \eqref{eq:FT}.
The formula \eqref{eq:fDPEhat} is similar to \cite[Corollary 3.2]{KieKos2016b}
that applies to the multiplicative setting where  the Mellin transform 
is used instead of the Fourier transform.  

We write $X \sim DPE(w)$ if $X$ is a random Hermitian matrix whose eigenvalues are
a polynomial ensemble as in \eqref{eq:polensembleDT} with function $w$. 
The following is now almost immediate. It is the analogue of  
\cite[Corollary 3.4]{KieKos2016b} in the additive setting.
\begin{corollary}
If $X$ and $Y$ are independent unitarily invariant random Hermitian matrices
whose eigenvalues are polynomial ensembles of derivative type 
(in the additive sense),
say $X \sim DPE(w_1)$ and $Y \sim DPE(w_2)$, then $X + Y$
is a random matrix whose eigenvalues are a polynomial ensemble
of derivative type (in the additive sense)
\[ X + Y \sim DPE(w_1 \ast w_2),  \]
where $\ast$ denotes the usual convolution of  functions on the real line
\[ w_1 \ast w_2(x) = \int_{-\infty}^{\infty} w_1(x-y) w_2(y) dy \]
\end{corollary}
\begin{proof}
This follows from Theorem \ref{thm:main},
\eqref{eq:fDPEhat}, and the basic properties of the Fourier transform. 
\end{proof}

Using \eqref{eq:fPEhat} and \eqref{eq:fDPEhat} in Theorem \ref{thm:main}
we obtain the following result.
\begin{corollary}
If $X$ and $Y$ are independent unitarily invariant random Hermitian matrices
with the eigenvalues of $X$ a polynomial ensembles of derivative type
(in the additive sense)
say $X \sim DPE(w)$ and the eigenvalues of $Y$ a polynomial ensemble
$Y \sim PE(w_1, \ldots, w_n)$, then $X + Y$
is a random matrix whose eigenvalues are a polynomial ensemble
\begin{equation} \label{eq:PEplusDPE} 
	X + Y \sim PE(w \ast w_1, \ldots, w \ast w_n). 
	\end{equation}
\end{corollary}
\begin{proof}
This is very similar, and we omit the proof. 
See \cite[Theorem 3.3]{KieKos2016b} for the analogous result 
in the multiplicative setting.
\end{proof}

\begin{remark}
If a function $w$ generates a polynomial ensemble \eqref{eq:polensembleDT2}
then clearly \eqref{eq:polensembleDT2} should be nonnegative for 
every choice of $x_1, x_2, \ldots, x_n$. This property is 
satisfied by so-called P\'olya frequency functions, see e.g.\ \cite{Kar1968}, 
for which it holds that
\[ \det \left[ w^{(k-1)}(x_j) \right] \geq 0 \]
whenever $x_1 < x_2 < \cdots < x_n$. Faraut \cite{Far2006} has
an interesting survey that connects P\'olya frequency functions
to the representation theory of Gelfand pairs, 
including $(\Un \ltimes \Hermn, \Un)$.
\end{remark}

\section{More on addition with LUE matrix}
\subsection{Alternative proof of Corollary \ref{cor:LUEsum}}
We give a different proof of Corollary \ref{cor:LUEsum}, based on a change 
of variables.
When integrating out the eigenvectors of $Y$ we encounter a matrix integral
over the unitary group that was recently evaluated 
by Kieburg, Kuijlaars and Stivigny in \cite{KieKuiSti2016}. 

Assume first that $X$ is fixed with eigenvalues $x_1, \ldots, x_n$ and $L$ is
an LUE matrix with parameter $\alpha > -1$.
Then $L \mapsto Y= X+L $ is a change of variables. 
From \eqref{eq:fLUE}  we arrive at the probability density
\[ \propto \det (Y-X)^{\alpha} e^{-\Tr (Y-X)} \] 
on the set of Hermitian matrices $Y$ with $Y \geq X$. Letting $y_1, \ldots, y_n$
be the eigenvalues of $Y$, this is 
\begin{equation} \label{eq:LUEsum1} 
	\propto \left( \prod_{j=1}^n e^{-y_j} \right) 
	\left( \prod_{k=1}^n e^{x_k} \right) 
	\det (Y-X)^{\alpha} \, \one_{Y \geq X}.
	\end{equation}

Introduce the eigenvalue decomposition $Y = U D U^*$ with a diagonal
matrix $D = \diag(y_1, \ldots, y_n)$ and a unitary matrix $U$.
The Jacobian of the eigenvalue decomposition is proportional to $\Delta_n(y)^2$,
and we obtain from \eqref{eq:LUEsum1} for the density of eigenvalues
\begin{equation} \label{eq:LUEsum2} 
	 \propto \Delta_n(y)^2   \left( \prod_{j=1}^n e^{-y_j} \right) 
	 \left( \prod_{k=1}^n e^{x_k} \right)
	\int\limits_{U \in \Un : U DU^* \geq X} \det(UDU^* - X)^{\alpha} \,  dU 
	\end{equation}
with a proportionality constant that does not depend on $X$.
The integral over the subset of the unitary group was evaluated in 
\cite[Theorem 2.3]{KieKuiSti2016} where it was found to be proportional to
\[   \frac{ \det \left[ (y_j - x_k)_+^{\alpha+n-1} \right]_{j,k=1}^n}
{\Delta_n(x) \Delta_n(y)}.   \]
Here $(y-x)_+ = \max(y-x,0)$.
The density \eqref{eq:LUEsum2} therefore is
\begin{equation} \label{eq:LUEsum3} 
	\propto  \frac{\Delta_n(y)}{\Delta_n(x)}  \left(\prod_{j=1}^n e^{-y_j}\right) 
		\left( \prod_{k=1}^n e^{x_k} \right)
	\det \left[ (y_j - x_k)_+^{\alpha+n-1} \right]_{j,k=1}^n.
	\end{equation}
We bring the prefactors into the determinant and obtain
\begin{equation} \label{eq:LUEsum4} 
	\propto  \frac{\Delta_n(y)}{\Delta_n(x)}
	\det \left[  (y_j - x_k)_+^{\alpha+n-1} e^{-y_j + x_k} \right]_{j,k=1}^n. 
	\end{equation}
This is the result for a fixed matrix $X$ and \eqref{eq:LUEsum4} shows that the
 eigenvalues of $Y = X+L$ are a polynomial ensemble.

Now suppose that $X$ is random and its eigenvalues are a polynomial ensemble
\eqref{eq:polensemble}. Then by the Andreief formula \eqref{eq:Andreief}, 
we find from \eqref{eq:LUEsum4} after averaging over \eqref{eq:polensemble},
\[ \propto \Delta_n(y) 
	\det \left[ \int_{-\infty}^{\infty}  (y_j-x)_+^{\alpha+n-1}  e^{-y_j+x} 
	f_k(x) dx \right]_{j,k=1}^n. \]
We change variables $x \mapsto y_j-x$ and arrive at
\begin{equation} \label{eq:LUEsum5} 
	\propto \Delta_n(y)  \det 
	\left[ \int_{0}^{\infty} x^{\alpha+n-1} e^{-x} f_k(y_j-x) dx
	 \right]_{j,k=1}^n
	\end{equation}
which is indeed a polynomial ensemble with the functions \eqref{eq:gkLUE} 	
as claimed in Corollary \ref{cor:LUEsum}.
	\qed

\subsection{Biorthogonal functions}
Corollary \ref{cor:LUEsum} is a transformation result  for polynomial
ensembles. 
A polynomial ensemble is a special case of a determinantal point process. The
correlation kernel for a polynomial ensemble
\[ \frac{1}{Z_n} \Delta_n(x) \, \det\left[ f_j(x_k) \right]_{j,k=1}^n \] 
takes the form
\[ \sum_{k=0}^{n-1} p_k(x) q_k(y) \]
where each $p_k$ is a monic  polynomial of degree $k$, each $q_k$ belongs to the
linear span of $f_1, \ldots, f_n$ and the biorthogonality condition
\begin{equation} \label{eq:biorthogonal} 
	\int_{-\infty}^{\infty} p_j(x) q_k(x) dx = \delta_{j,k} 
	\end{equation}
holds, see \cite{Bor1998,For2010}.

Suppose $Y = X + L$ where $X \sim PE(f_1, \ldots, f_n)$ and 
$Y \sim PE(g_1, \ldots, g_n)$ as in Corollary \ref{cor:LUEsum}. We write $K_n^X$ 
and $K_n^Y$ for the correlation kernels of the two polynomial
ensembles and 
\begin{align} \label{eq:KnX} 
	K_n^X(x,y) & = \sum_{k=0}^{n-1} p_k(x) q_k(y), \\
    K_n^Y(x,y) & = \sum_{k=0}^{n-1} P_k(x) Q_k(y), 
    \end{align}
and we investigate the relation between the two sets of biorthogonal functions.
The transformation results are in formulas \eqref{eq:Qkdef}, \eqref{eq:Pkintegral},
and \eqref{eq:KnY} below.
See \cite{ClaKuiWan2015} for similar results related to addition with a GUE matrix.

We assume $p_k$ and $q_k$ are given. We fix $Q_k$ by taking
\begin{equation} \label{eq:Qkdef} 
	Q_k(y) = \frac{1}{\Gamma(\alpha+n)} \int_{0}^{\infty}  
	x^{\alpha+n-1} e^{-x} q_k(y-x) \, dx, 
	\qquad k = 0, \ldots, n-1. \end{equation}
Then the $Q_k$ are in the linear span of $g_1, \ldots, g_n$
because of \eqref{eq:gkLUE} and the fact that each $q_k$ is in the linear span 
of $f_1, \ldots, f_n$.
We want to find polynomials $P_0, \ldots, P_{n-1}$ such that
\[ \int_{-\infty}^{\infty} P_j(x) Q_k(x) dx = \delta_{j,k}, \]
and $\deg P_j = j$ for $j =0, \ldots, n-1$.

We calculate from \eqref{eq:Qkdef} (for an as yet unknown $P_j$)
\[ \int_{-\infty}^{\infty} P_j(x) Q_k(x) dx 
	= \frac{1}{\Gamma(\alpha+n)} \int_{-\infty}^{\infty} 
		\int_0^{\infty} t^{\alpha+n-1} e^{-t} P_j(x) q_k(x-t) \, dt \, dx \]
and make the changes of variables $t' = x-t$, $s' = t$. Then by Fubini's theorem
\begin{equation} \label{eq:PjQkintegral} 
	\int_{-\infty}^{\infty} P_j(x) Q_k(x) dx 
	= \frac{1}{\Gamma(\alpha+n)} \int_{-\infty}^{\infty} \left(
		\int_0^{\infty} s^{\alpha+n-1} e^{-s} P_j(t+s) ds \right)  q_k(t) \, dt. 
		\end{equation}
We want to choose $P_k$ such that
\begin{equation} \label{eq:Pjwant} 	
	\frac{1}{\Gamma(\alpha+n)} \int_0^{\infty} s^{\alpha+n-1} e^{-s} P_k(x+s) ds 
	= p_k(x) 
	\end{equation}
for $k=0, 1, \ldots, n-1$, since then by \eqref{eq:PjQkintegral}
and the biorthogonality for the $p_j$'s and $q_k$'s,
\[ \int_{-\infty}^{\infty} P_j(x) Q_k(x) dx = \int_{-\infty}^{\infty} p_j(t) 
	q_k(t) dt = \delta_{j,k}. \]

The mapping $\mathcal L : f \mapsto \mathcal Lf$ with
\begin{equation} \label{eq:Lfx}
	\mathcal Lf(x) = \frac{1}{\Gamma(\alpha+n)} \int_0^{\infty} s^{\alpha+n-1} 
	e^{-s} f(x+s) ds 
	\end{equation}
maps polynomials to polynomials of the same degree and the same leading
coefficient.  
Consider the polynomials $e_k(x) = \frac{1}{k!} x^k$. We have by \eqref{eq:Lfx}
and the binomial theorem
\begin{align*} 
	\mathcal L e_k(x) & = \frac{1}{k! \Gamma(\alpha+n)} 
	\int_0^{\infty} s^{\alpha+n-1} e^{-s} (x+s)^k ds  \\
	&= \sum_{j=0}^k \frac{1}{j! (k-j)! \Gamma(\alpha+n)} 
	\int_0^{\infty} s^{\alpha+n-1} e^{-s} x^{k-j} s^j ds \\
	& = \sum_{j=0}^k \frac{\Gamma(\alpha+n+j)}{j! (k-j)! \Gamma(\alpha+n)} 
	  x^{k-j}  \\
	& = \sum_{j=0}^k \binom{\alpha+n+j-1}{j} e_{k-j}(x) 
	\end{align*}
Since $e_{k-j} = e_k^{(j)}$, we conclude that 
\begin{equation} \label{eq:Lfx2} 
	\mathcal L f = \sum_{j=0}^{\infty} \binom{\alpha+n+j-1}{j} f^{(j)} 
	\end{equation}
if $f$ is one of the functions $e_k$, and then by linearity for 
arbitrary polynomials $f$. Note that \eqref{eq:Lfx2} is a
finite sum if $f$ is a polynomial.

To invert $\mathcal L$ we need a sequence $(a_k)_k$ such that $a_0 = 0$ and
\begin{equation} \label{eq:Linv1}  
	\sum_{j=0}^k a_{k-j} \binom{\alpha + n + j-1}{j} = 0, \qquad k = 1,2, \ldots, 
	\end{equation}
since then we can put
\begin{equation} \label{eq:Linv2} 
	\mathcal L^{-1} f = \sum_{k=0}^{\infty} a_k f^{(k)}. \end{equation}
The recurrence \eqref{eq:Linv1} is solved by the numbers
\begin{equation} \label{eq:akdef} 
	a_k =  (-1)^k \binom{\alpha+n}{k}. 
	\end{equation}
Indeed, note that
\[ (-1)^j \binom{\alpha+n+j-1}{j} = \binom{-\alpha-n}{j} \]
and then
\begin{align*} 
	\sum_{j=0}^k (-1)^{k-j} \binom{\alpha+n}{k-j} \binom{\alpha + n + j-1}{j} & = 
	(-1)^k \sum_{j=0}^k \binom{\alpha+n}{k-j} \binom{-\alpha-n}{j} \\
	& = (-1)^k \binom{0}{k} \end{align*}
by the Chu-Vandermonde identity.

The conclusion from \eqref{eq:Lfx}, \eqref{eq:Linv2}
and \eqref{eq:akdef}  is that 
\begin{equation} \label{eq:Pkdef} 
	P_k(x) = \sum_{j=0}^{k} (-1)^j \binom{\alpha+n}{j} p_k^{(j)}(x), 
	\end{equation}
is indeed a monic polynomial of degree $k$ that satisfies \eqref{eq:Pjwant}
and the biorthogonality property
\[ \int_{-\infty}^{\infty} P_j(x) Q_k(x) dx = \delta_{j,k}, 
	\qquad j,k=0, \ldots, n-1 \]
holds.	

There is a contour integral formula for $P_k$
\begin{equation} \label{eq:Pkintegral} 
	P_k(x) =  \frac{\Gamma(\alpha+n+1)}{2\pi i}
	\int_C \frac{p_k(x-s)}{s^{\alpha+n+1}} e^s ds \end{equation}
where $C$ is a contour encircling the negative real axis,
starting at $-\infty$ in the lower half plane, and
ending at $-\infty$  in the upper half plane. [It can be taken to
be closed contour around the origin if $\alpha$ is an integer.]
To prove the integral
formula \eqref{eq:Pkintegral} we define a mapping $\mathcal M$ by 
\[ \mathcal M f (x) = \frac{\Gamma(\alpha+n+1)}{2\pi i}
	\int_C \frac{f(x-s)}{s^{\alpha+n+1}} e^s ds \]
and we evaluate $\mathcal M e_k$ where $e_k(x) = \frac{1}{k!} x^k$ as before. 
By the binomial theorem
\[ \mathcal Me_k = \frac{\Gamma(\alpha+n+1)}{2\pi i k!}
	\int_C \sum_{j=0}^k (-1)^j x^{k-j} \binom{k}{j} \frac{1}{s^{\alpha+n-j+1}} 
	e^s ds \]
which is
\[ \sum_{j=0}^k \frac{(-1)^j}{j!} e_{k-j} 
	\frac{\Gamma(\alpha+n+1)}{\Gamma(\alpha+n-j+1)} \] 
and this is
\[ \sum_{j=0}^k (-1)^j \binom{\alpha+n}{j} e_k^{(j)}. \]
By linearity we have
\[ \mathcal M f(x) = \sum_{j=0}^{\infty} (-1)^j \binom{\alpha+n}{j} f^{(j)} \]
for every polynomial $f$, which indeed coincides with the formula 
\eqref{eq:Linv2}-\eqref{eq:akdef} 
for $\mathcal L^{-1}$.

For the correlation kernels we arrive at the transformation formula
\begin{align} \nonumber 
	K_n^Y(x,y) & = \sum_{k=0}^{n-1} P_k(x) Q_k(y) \\
	& = \frac{\alpha +n}{2\pi i} \int_C
		\int_0^{\infty}  \left(\frac{t}{s} \right)^{\alpha+n} e^{s-t} 
		K_n^X(x-s,y-t) \frac{dt}{t} \frac{ds}{s}, 
	\label{eq:KnY}	
\end{align}
which could be useful for asymptotic analysis. A similar formula for the case of 
a sum with
a GUE matrix was given in \cite{ClaKuiWan2015} and it was used for asymptotic
 analysis in \cite{ClaDoe2016} and \cite{ClaKuiLieWan2016}.

\paragraph{Acknowledgements}
The first author is very grateful to Mario Kieburg and Holger K\"osters 
for inspiring
discussions related to their papers \cite{KieKos2016a,KieKos2016b}. 

Arno Kuijlaars is supported by long term structural funding-Methusalem 
grant of the Flemish Government, by the Belgian Interuniversity Attraction 
Pole P07/18,  by KU Leuven Research Grant OT/12/073, and by FWO Flanders 
projects G.0934.13 and G.0864.16.

The research of Pablo Rom\'an is supported by the
Radboud Excellence Fellowship. P. Rom\'an was partially supported by CONICET
grant PIP 112-200801-01533, FONCyT grant PICT 2014-3452 and by SECyT-UNC.

\end{document}